%% file: GA15OnlineVersion.tex
\documentclass[11pt]{amsart}
\usepackage{amscd}
\usepackage{mathrsfs}
\usepackage[latin1]{inputenc}
\usepackage{amssymb,latexsym,color}
\usepackage{type1cm,ae}
\usepackage{graphicx}
\usepackage{xspace}
\usepackage{enumerate}
\usepackage{setspace}
\usepackage[english]{babel}
\newtheorem{theorem}{Theorem}[section]
\newtheorem{lemma}[theorem]{Lemma}

\theoremstyle{definition}
\newtheorem{definition}[theorem]{Definition}
\newtheorem{example}[theorem]{Example}

\DeclareMathOperator{\diam}{diam}

\DeclareMathOperator{\capacity}{Cap}

\newcommand*{\bR}{\ensuremath{\mathbb{R}}}
\newcommand*{\bN}{\ensuremath{\mathbb{N}}}

\newcommand*{\bdary}[1]{\partial #1}

\newcommand\numberthis{\addtocounter{equation}{1}\tag{\theequation}}

\def\Xint#1{\mathchoice
  {\XXint\displaystyle\textstyle{#1}}%
  {\XXint\textstyle\scriptstyle{#1}}%
  {\XXint\scriptstyle\scriptscriptstyle{#1}}%
  {\XXint\scriptscriptstyle\scriptscriptstyle{#1}}%
  \!\int}
\def\XXint#1#2#3{{\setbox0=\hbox{$#1{#2#3}{\int}$}
  \vcenter{\hbox{$#2#3$}}\kern-.5\wd0}}

\def\dashint{\Xint-}

\theoremstyle{remark}
\newtheorem{remark}[theorem]{Remark}

\numberwithin{equation}{section}

\newcommand{\abs}[1]{\lvert#1\rvert}

\begin{document}

\title[Regularity of the boundary extension in uniform domains]{Mappings of finite distortion: boundary extensions in uniform domains}





\author{Tuomo \"Akkinen}
\address[Tuomo \"Akkinen]{Department of Mathematics and Statistics, University of Jyv\"askyl\"a, P.O. Box 35, FI-40014 University of Jyv\"askyl\"a, Finland}
\email{tuomo.akkinen@jyu.fi}	

\author{Chang-Yu Guo}
\address[Chang-Yu Guo]{Department of Mathematics and Statistics, University of Jyv\"askyl\"a, P.O. Box 35, FI-40014 University of Jyv\"askyl\"a, Finland}
\email{changyu.c.guo@jyu.fi}
\thanks{C.Y.Guo was supported by the Magnus Ehrnrooth foundation.}


\begin{abstract}
In this paper, we consider mappings on uniform domains with exponentially integrable distortion whose Jacobian determinants are integrable. We show that such mappings can be extended to the boundary and moreover these extensions are exponentially integrable with quantitative bounds. This extends previous results of Chang and Marshall~\cite{CM} on analytic functions, Poggi-Corradini and Rajala~\cite{PR} and \"Akkinen and Rajala~\cite{Ae2} on mappings of bounded and finite distortion.
\end{abstract}

\maketitle

\section{Introduction}
A mapping $f\colon\Omega\to\bR^n$, on a domain $\Omega\subset\bR^n$ has finite distortion if the following conditions are fulfilled:
\begin{enumerate}[$(a)$]
 \item $f\in W^{1,1}_\text{loc}(\Omega,\bR^n)$,
 \item $J_f=\textup{det}(Df)\in L^1_\text{loc}(\Omega)$,
 \item there exists a measurable $K_f\colon\Omega\to[1,\infty)$, so that for almost every $x\in\Omega$ we have \[\abs{Df(x)}^n\leq K_f(x)J_f(x), \]
\end{enumerate}
where $|\cdot|$ is the operator norm.  If $K_f \leq K < \infty$ almost everywhere, we say that $f$ is $K$-quasiregular. If $n=2$ and $K=1$, we recover complex analytic functions. See \cite{Re}, \cite{Ri} and \cite{Vu} for the theory of quasiregular mappings, and \cite{HK}, \cite{IM} for the theory of mappings of finite distortion.

Let $\Omega$ be a bounded subdomain of $\bR^n$ with a distinguished point $x_0\in \Omega$. We will consider the collection of mappings of finite distortion $f\colon\Omega\to\bR^n$ satisfying $f(x_0)=0$ and
\begin{equation}\label{jacobi}
\int_\Omega J_f(x)\,dx\leq \abs{\Omega},
\end{equation}
where $\abs{\Omega}$ denotes the Lebesgue $n$-measure of $\Omega$. Moreover, we assume that there are constants $\lambda, \mathcal{K}>0$ such that 
\begin{equation}
\label{perusexp}
\int_{\Omega}\exp(\lambda K_f)\,dx\leq\mathcal{K}. 
\end{equation}
Let $\mathcal{F}_{\lambda,\mathcal{K}}(\Omega)$ denote the class of mappings satisfying \eqref{jacobi} and \eqref{perusexp}. The class of $K$-quasiregular mappings with \eqref{jacobi} will be denoted by $\mathcal{F}_{K}(\Omega)$.

If $\mathcal{F}$ is the class of analytic functions on the unit disc $\mathbb{D}$ with $f(0)=0$ and 
\begin{equation*}
	\int_\mathbb{D}\abs{f'(z)}^2\,dA(z)\leq \pi.
\end{equation*}
In \cite[Corollary 1]{CM} Chang and Marshall proved the following sharp extension of an earlier result by Beurling \cite{B}, that is 
\begin{equation}\label{Moser}
\sup_{f\in\mathcal{F}}\int_0^{2\pi} \exp(\abs{\bar{f}(e^{i\theta})}^2)\,d\theta<\infty. 
\end{equation}
This result is sharp, since for Beurling functions $B_a\colon\mathbb{D}\to\mathbb{C}$, $0<a<1$,
$$
B_a(z)=\log\left(\frac{1}{1-az}\right)\log^{-1/2}\left(\frac{1}{1-a^2}\right)
$$
one can show that
$$
\lim_{a \to 0} \int_0^{2\pi} \exp(\gamma\abs{B_a(e^{i\theta})}^2)\,d\theta=\infty
$$
for every $\gamma>1$. After this many natural generalizations of \eqref{Moser} have been proved. First, Marshall \cite{Ma} gave another proof for \eqref{Moser} using different techniques than in \cite{CM}. Poggi-Corradini \cite{P} was able to shorten the proof of Marshall \cite{Ma} using so called ``egg-yolk principle". Poggi-Corradini and Rajala proved the ``egg-yolk principle" for quasiregular mappings in space and with this they showed in \cite[Theorem 1.1]{PR} that for $n\geq 2$ we have 
\begin{equation}\label{Qrmoser}
\sup_{f\in\mathcal{F}_K(\mathbb{B}^n)}\int_{\mathbb{S}^{n-1}} \exp(\alpha\abs{\bar{f}(\xi)}^\frac{n}{n-1})\,d\mathcal{H}^{n-1}(\xi)<\infty, 
\end{equation}
where the constant 
$$\alpha=(n-1)\left(\frac{n}{2K}\right)^\frac{1}{n-1}$$ 
is sharp in dimension $n=2$. Here $\mathbb{B}^n$ and $\mathbb{S}^{n-1}$ denote the unit ball and its boundary, respectively. Note that mappings $f\in \mathcal{F}_{\lambda,\mathcal{K}}(\mathbb{B}^n)$ can be extended to $\partial\mathbb{B}^n=\mathbb{S}^{n-1}$ using the fact that $f$ has radial limits a.e. on $\mathbb{S}^{n-1}$.
Recently it was shown by \"Akkinen and Rajala \cite{Ae2} that \eqref{Qrmoser} extends to the class $\mathcal{F}_{\lambda,\mathcal{K}}(\mathbb{B}^n)$ in the form
\begin{equation}\label{MOFDmoser}
\sup_{f\in\mathcal{F}_{\lambda,\mathcal{K}}(\mathbb{B}^n)}\int_{\mathbb{S}^{n-1}} \exp(\beta\abs{\bar{f}(\xi)})\,d\mathcal{H}^{n-1}(\xi)<\infty, 
\end{equation}
where $\beta=\beta(n,\mathcal{K}, \lambda)>0$. This result is sharp in terms of the exponent $1$ but the sharp constant $\beta$ is not known even in dimension $n=2$, see discussion in \cite[Section 4]{Ae2}. 

The purpose of this paper is to prove  \eqref{Qrmoser} and \eqref{MOFDmoser} for more general domains than the unit ball.  We will consider domains $\Omega$ that are uniform, the justification for this restriction is given in Section 3. The boundary of a uniform domain is known to have Hausdorff dimension strictly less than $n$ but for technical reasons we will assume more regularity than this. If $\nu$ is a positive Borel measure we say that $\partial\Omega$ is upper $s$-Ahfors regular with respect to $\nu$ if there exist constants $s\in(0,n)$ and $C>0$ such that for any $x\in\Omega$
\begin{equation}\label{sreg}
\nu\left(\partial\Omega\cap B(x,r)\right)\leq C r^s\quad\text{for}~0<r\leq\diam(\Omega).
\end{equation}
Our main theorems are the following:

\begin{theorem}\label{thm:regularity of boundary extension for qr}
Let $\Omega\subset \bR^n$, $n\geq 2$, be a $c$-uniform domain with center $x_0$ and $\nu$ any positive Borel measure satisfying \eqref{sreg}. Then there exists a constant $\alpha=\alpha(\Omega,n, K,c,s)>0$ such that 
\begin{align*}
\sup_{f\in \mathcal{F}_{K}(\Omega)}\int_{\bdary\Omega}\exp\Big(\alpha|\bar{f}(w)|^{n/(n-1)}\Big)d\nu(w)<\infty,
\end{align*} 
where $\bar{f}$ is the extension of $f$ to $\partial\Omega$ given in Definition~\ref{def:natural boundary extension}.
\end{theorem}

\begin{theorem}\label{thm:regularity of boundary extension}
Let $\Omega\subset \bR^n$, $n\geq 2$, be a $c$-uniform domain with center $x_0$ and $\nu$ any positive Borel measure satisfying \eqref{sreg}. Then there exists a constant $\beta=\beta(\Omega,n,\mathcal{K},\lambda,c,s)>0$ such that 
\begin{align*}
\sup_{f\in \mathcal{F}_{\lambda,\mathcal{K}}(\Omega)}\int_{\bdary\Omega}\exp\left(\beta|\bar{f}(w)|\right)d\nu(w)<\infty,
\end{align*} 
where $\bar{f}$ is the extension of $f$ to $\partial\Omega$  given in Definition~\ref{def:natural boundary extension}.
\end{theorem}

These results are sharp with respect to the exponent inside the exponential map since $\mathbb{B}^n$ is a particular example of a uniform domain. The sharp constants $\alpha$ and $\beta$ remain open, we would need new techniques for this problem. The general idea behind the proof is similar to the one used in the previous papers \cite{Ae2,PR}, essentially relying on a version of a Theorem of Beurling \cite{B} originally stated for analytic functions. Allowing more general domains produces a two-fold problem: first we have to find a way to extend our mapping to the boundary and also we need a new version of the Beurling Theorem.

It could be possible to deduce Theorem \ref{thm:regularity of boundary extension} from earlier works by Cianchi \cite{Ci} and \cite{Ci2} as follows: one first uses the Orlicz-Sobolev extension in uniform domains for each of the component functions $f_i$, $i=1,\dots,n$, to obtain a global Orlicz-Sobolev function $f_i':\bR^n\to \bR$; then we can choose a quasicontinuous representative of $f_i'$, which means outside a set of small capacity, $f_i'$ is continuous. This in return gives us a natural way to extend our original function $f_i$ to the boundary of $\Omega$ as $\bar{f}_i:\bdary\Omega\to \bR$. Now extending results for Orlicz-Sobolev functions in \cite{Ci2} to similar ones given for Lorentz-Sobolev functions in \cite[Theorems 2.4 and 2.5]{Ci} and using above arguments one can show that Theorem~\ref{thm:regularity of boundary extension} is valid for each component function $\bar{f}_i$. So Theorem~\ref{thm:regularity of boundary extension} follows once we set $\bar{f}=(\bar{f}_1,\dots,\bar{f}_n)$.

In this way, one could prove Theorem~\ref{thm:regularity of boundary extension} with a constant $\alpha$ independent of $\lambda$ and $\mathcal{K}$ and thus cannot obtain the sharp constant since the sharp constant depends on the distortion $K$; see~\cite{PR}. We believe that the techniques used here can be further extended to more irregular domain than uniform domains since most of our results can be generalized to domains satisfying \textit{quasihyperbolic boundary condition} (cf.~\cite{kot01}).

This paper is organized as follows. Section 2 fixes the notation. In section 3 we will show that $f$ can be extended to the boundary of a uniform domain along John-curves and that this extension is well-defined with a small exceptional set on $\partial\Omega$ in terms of Hausdorff gauges. Section 4 contains the main ingredients of the proofs of Theorems  \ref{thm:regularity of boundary extension for qr} and \ref{thm:regularity of boundary extension}, which is the capacity estimate given in Theorem \ref{thm:main thm}. Finally in Section 5 we will prove our main Theorems.




\section{Notation and preliminaries}
A curve $\gamma$ in $\Omega$ is a continuous mapping $\gamma\colon[0,1]\to\Omega$. The trace $\gamma([0,1])$ of a curve $\gamma$ is denoted by $\abs{\gamma}$. A curve $\gamma$ is said to connect points $x,y\in\Omega$ if $\gamma(0)=y$ and $\gamma(1)=x$ and points $x\in\Omega$, $y\in\partial\Omega$ if $\gamma(1)=x$, $\gamma((0,1])\subset\Omega$ and 
$$\lim_{t\to 0^+}\gamma(t)=y.$$
We also use the notation $l(\gamma)$ to denote the (euclidean) length of a curve $\gamma$.

The quasihyperbolic metric $k_\Omega$ in a domain $\Omega\subsetneq\bR^n$ is defined as
\begin{equation*}
    k_\Omega(x,y)=\inf_{\gamma}\int_\gamma \frac{ds}{d(z,\bdary\Omega)},
\end{equation*}
where the infimum is taken over all rectifiable curves $\gamma$ in $\Omega$ which join $x$ to $y$. This metric was introduced by Gehring and Palka in~\cite{gp76}. A curve $\gamma$ connecting $x$ to $y$ for which 
$$k_\Omega(x,y)=\int_\gamma \frac{ds}{d(z,\bdary\Omega)}$$ is called a quasihyperbolic geodesic. Quasihyperbolic geodesics joining any two points of a proper subdomain of $\bR^n$ always exists; see \cite[Lemma 1]{go79}. Given two points $x,y\in \Omega$, we denote by $[x,y]$ any quasihyperbolic geodesic that joins $x$ and $y$.


Let $\Omega\subset\bR^n$ be a domain. We will denote by $\mathscr{W}:=\mathscr{W}(\Omega)$ a Whitney type decomposition of $\Omega$, where $\mathscr{W}$ is a collection of closed cubes $Q\subset \Omega$ with pairwise disjoint interiors and having edges parallel to coordinate axis, such that 
$$\Omega=\bigcup_{Q\in\mathscr{W}}Q.$$ 
Moreover, there exists $\lambda_0>1$ such that the cubes dilated by factor $\lambda_0$ have uniformly bounded overlap, the diameters of $Q\in\mathscr{W}$ take values in $\{2^{-j}:j\in\mathbb{Z}\}$ and satisfy condition
$$\diam{Q}\leq d(Q,\partial\Omega) \leq 4 \diam{Q}.$$
We will also use the notation $$\mathscr{W}_j=\left\{Q\in\mathscr{W}:\diam{Q}=2^{-j}\right\}.$$

\section{Boundary extensions}\label{sec:boundary extension}
In the previous papers \cite{Ma}, \cite{PR}, \cite{PPR} and \cite{Ae2} the boundary extension of $f\in\mathcal{F}_{\lambda,\mathcal{K}}(\mathbb{B}^n)$ was defined via the radial limits of $f$ that exist at almost every point in $\mathbb{S}^{n-1}$ by absolute continuity on lines, \eqref{jacobi} and \eqref{perusexp}. In this section we prove similar results for more general class of  domains, namely for uniform domains.


\subsection{Existence of limits along curves}\label{subsec:exist limits}
When considering general bounded domains $\Omega$ there might be big parts of the boundary $\partial \Omega$ that are not accessible by a rectifiable curve inside  $\Omega$, see Figure \ref{fig:topocomb}.

\begin{center}
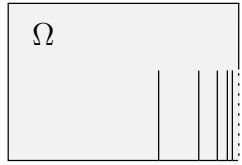
\begin{figure}[here]
\input{Topocomb}
\caption{Topologist comb}
\label{fig:topocomb}
\end{figure}
\end{center}

To avoid such situations, we restrict our considerations to John domains.
\begin{definition}
A bounded domain $\Omega\subset \bR^n$ is a $c$-John domain with distinguished point $x_0\in \Omega$ (center) if every point $x\in\Omega$ can be joined to $x_0$ by a $c$-John curve. Recall also that a rectifiable curve $\gamma\colon\mathopen[0,l\mathclose]\to\Omega$ parametrized by arc length is called a $c$-John curve joining $x$ to $x_0$ if it satisfies $\gamma(0)=x$, $\gamma(l)=x_0$ and
\begin{equation}\label{John}
d\left(\gamma(t),\partial\Omega\right)\geq\frac{1}{c}t
\end{equation}
for every $t\in\mathopen[0,l\mathclose]$.
\end{definition}
It follows from the lower semicontinuity of $l(\gamma)$ and Ascoli-Arzela Theorem that each boundary point can be connected to $x_0$ by a curve satisfying \eqref{John}. Thus every point on the boundary is accessible, even in a non-tangential manner. Notice that typically there will be uncountably many $c$-John curves between $x_0\in\Omega$ and $y\in\partial\Omega$. For the rest of this subsection the standing assumptions are: Let $\Omega$ be a $c_0$-John domain with center $x_0$, $\mathscr{W}$ a Whitney decomposition of $\Omega$ and $c\geq c_0$.

If  $\xi\in\partial\Omega$ then define $\textup{I}_c(\xi,x_0)$ as  the collection of all $c$-John curves connecting $\xi$ to $x_0$ and 
$$P_c(\xi)=\left\{Q\in\mathscr{W}:Q\cap\gamma\neq\emptyset~\text{for some $\gamma\in \textup{I}_c(\xi,x_0)$}\right\}.$$
For $Q\in\mathscr{W}$ and $E\subset \bdary\Omega$, we define the $c$-shadow of $Q$ on $E$ by
$$S_E^c(Q)=\left\{\xi\in E : Q\in P_c(\xi)\right\}.$$ 
When $E=\bdary\Omega$, we write $S^c(Q)$ instead of $S_{\bdary\Omega}^c(Q)$. We will omit the sub/superscripts $c$ in the above notation if $c=c_0$.

The following two lemmas are well-known and the proofs can be found, for instance, in ~\cite[Section 4]{G13}. 

\begin{lemma}\label{ShadowProperties}
Let $Q\in\mathscr{W}$ and  $\xi\in\partial\Omega$. Then $S^c(Q)$ is closed and there exists a constant $C(c)>0$ such that
$$\diam{S^c(Q)}\leq C(c)\diam{Q}.$$
Also there exists a constant $C(n,c)>0$ such that for any $j\in\mathbb{Z}$ 
$$\#\left\{Q\in\mathscr{W}_j:\xi\in S^c(Q)\right\}\leq C(n,c).$$
\end{lemma}

\begin{lemma}\label{JohnMeas}
Assume that $\mu$ is a Borel measure on $\partial\Omega$ and $E\subset \partial\Omega$ measurable. Then there exists a constant $C(n,c)>0$ such that, for each $j\in\mathbb{Z}$ we have
$$\sum_{Q\in\mathscr{W}_j}\mu\left(S_E^c(Q)\right)\leq C\mu(E).$$
\end{lemma}


Our aim is to prove that mappings in the classes $\mathcal{F}_K(\Omega)$ and $\mathcal{F}_{\lambda,K}(\Omega)$ have limits along $c_0$-John curves emanating from the center $x_0$. This will be true for big pieces of the boundary in terms of Hausdorff dimension. For this we introduce a notion of discrete length of a curve: Assume that $\xi\in\partial\Omega$ and $\gamma\in\textup{I}\mathopen(\xi,x_0\mathclose)$. Given a continuous mapping $f\colon\Omega\to\bR^n$ we define the discrete length of $f(\gamma)$ by
$$\ell_d \left[f(\gamma)\right]=\sum_{\substack{Q\in\mathscr{W}\\ Q\cap \gamma\neq\emptyset}} \diam{f(Q)}.$$
It follows from the definition and Lemma \ref{ShadowProperties} that if $\ell_d \left[f(\gamma)\right]<\infty$ then the limit
$$\lim_{t\to 0^+}f(\gamma(t))\in\mathbb{R}^n$$
exists.

We also need oscillation estimates for mappings of finite distortion $f\colon\Omega\to\bR^n$. The proof of the next lemma is contained in~\cite{KN}, see also~\cite{Ae}.
\begin{lemma}\label{oscilla}
Let $f\colon\Omega\to\bR^n$ be a mapping of finite distortion and assume $\sigma>1$. Then we have the following:
\begin{enumerate}[i)]
 \item If $f$ is quasiregular, then 
 $$\diam{f(B)}\leq C\left(\int_{\sigma B} J_f(x)\,dx\right)^{1/n}$$
 for every $B=B(x,r)$ for which $\sigma B\subset\subset\Omega$. Here constant $C>0$ depends only on $n$ and $K$. 
 \item If $f$ has exponentially integrable distortion, then
 $$\diam{f(B)}\leq C\left(\int_{\sigma B} J_f(x)\,dx\right)^{1/n}\log\left(1/\diam{B}\right)^{1/n}$$
 for every $B=B(x,r)$ for which $\sigma B\subset\subset\Omega$. Here constant $C>0$ depends only on $n$, $\lambda$ and $\int_\Omega\exp(\lambda K)\,dx$.
\end{enumerate}
\end{lemma}

Now we are ready to state a theorem which tells that a mapping of finite distortion can be extended to $\partial \Omega$ along John curves with a small exceptional set.


\begin{theorem}\label{thm:extension}
Assume that $f\colon\Omega\to\bR^n$ is a mapping of finite distortion with $J_f\in L^1(\Omega)$. Let $E_f$ be the set of points $w\in\partial\Omega$ for which there exists a curve $\gamma\in \textup{I}\mathopen(w,x_0\mathclose)$ so that $f$ does not have a limit along $\gamma$. Then we have the following:
\begin{enumerate}[i)]
 \item If $f$ is quasiregular and $h$ is a doubling gauge function satisfying
 \begin{equation}\label{qrgauge}
 \int_0 h(t)^{1/(n-1)}\frac{\,dt}{t}<\infty
 \end{equation}
 then $\mathcal{H}^h(E_f)=0$.
 \item If $f$ has exponentially integrable distortion and $h$ is a doubling gauge function satisfying
 \begin{equation}\label{mofdgauge}
 \int_0 \big(h(t)\log\left(1/t\right)\big)^{1/(n-1)}\frac{\,dt}{t}<\infty
 \end{equation}
 then $\mathcal{H}^h(E_f)=0$. 
\end{enumerate}
\end{theorem}

\begin{proof}
We will follow the proof of Theorem 1.1 in \cite{KN} and actually prove that $\mathcal{H}^h(A_\infty)=0$, where $A_\infty$ is the set of points in $\xi\in \partial\Omega$ for which there is a curve $\gamma\in \textup{I}\mathopen(\xi,x_0\mathclose)$ so that $f(\gamma)$ has infinite discrete length. Note that $E_f\subset A_\infty$. On the contrary, assume that $\mathcal{H}^h(A_\infty)>0$. Then $\mathcal{H}^h(A_k)>0$, where $A_k$ is the set of points in $\xi\in\partial\Omega$ for which there exist $\gamma\in \textup{I}\mathopen(\xi,x_0\mathclose)$ so that 
$\ell_d\left[f(\gamma)\right]\geq k.$
Then by Frostman's lemma there exists a Borel measure $\mu$ supported in $A_k$ so that
$$\mu\big(B^n(x,r)\big)\leq h(r)$$
for every $B^n(x,r)\subset\bR^n$ and
$$\mu(E_k)\simeq \mathcal{H}^h_\infty(A_k)\geq\mathcal{H}_\infty^h(A_\infty)>0.$$
By the definition of $A_k$
\begin{align*}
\mu(A_k)k&\leq\int_{A_k}\ell_d\left[f(\gamma_w)\right]\,d\mu(w)\\
&\leq\int_{A_k}\sum_{\substack{Q\in\mathscr{W}\\ Q\cap\gamma_w\neq\emptyset}}\diam{f(Q)}\,d\mu(w).
\end{align*}

From now on, we assume that $f$ has exponentially integrable distortion. Let $\lambda_0>1$ be the constant in the definition of Whitney decomposition $\mathscr{W}$ for which the cubes $q\in\mathscr{W}$ dilated by a factor $\lambda_0$ have uniformly bounded overlap. If $Q\in \mathscr{W}$ then set $B_Q=B^n(x_Q,r_Q)$, where $x_Q$ is the center of $Q$ and $r_Q=\diam{Q}/2$. Using Lemma \ref{oscilla} for balls $B_Q$ we have the following chain of inequalities:
\begin{align*}
\mu(A_k)k&\leq \int_{A_k}\sum_{\substack{Q\in\mathscr{W}\\ Q\cap\gamma_w\neq\emptyset}}\diam{f(B_Q)}\,d\mu(w)\\
&\leq\sum_{Q\in\mathscr{W}}\int_{A_k}\chi_{S(Q)}(w)\,d\mu(w)\diam{f(B_Q)}\\
&\leq\sum_{Q\in\mathscr{W}}\mu\left(S_{A_k}(Q)\right)\diam{f(B_Q)}\\
&\leq C\sum_{Q\in\mathscr{W}}\mu\left(S_{A_k}(Q)\right)\log^{\frac{1}{n}}\left(\frac{1}{\diam{B_Q}}\right)\\
&\times\left(\int_{\lambda_0 B_Q}J_f(x)\,dx\right)^{1/n}\\
&\leq C\left(\sum_{Q\in\mathscr{W}}\mu\left(S_{A_k}(Q)\right)^{\frac{n}{n-1}}\log^{\frac{1}{n-1}}\left(\frac{1}{\diam{B_Q}}\right)\right)^{(n-1)/n}\\
&\times\left(\sum_{Q\in\mathscr{W}}\int_{\lambda_0 B_Q}J_f(x)\,dx\right)^{1/n}\\
&\leq C\left(\sum_{j=1}^\infty\sum_{Q\in\mathscr{W}_j}\mu\left(S_{A_k}(Q)\right)^{n/(n-1)}j^{1/(n-1)}\right)^{(n-1)/n}.
\end{align*}
Above we also used the fact that $\diam{B_Q}\simeq 2^{-j}$ for $Q\in\mathscr{W}_j$ and also the property that $\lambda_0 Q$'s have uniformly bounded overlap. The left hand side is approximated using Lemmas \ref{ShadowProperties} and \ref{JohnMeas} together with the doubling property of gauge function $h$
\begin{align*}
\sum_{Q\in\mathscr{W}_j}\mu\left(S_{A_k}(Q)\right)^{n/(n-1)}&\leq\max_{Q\in\mathscr{W}_k}\mu\left(S_{A_k}(Q)\right)^{1/(n-1)}\sum_{Q\in\mathscr{W}_j}\mu\left(S_{A_k}(Q)\right)\\
&\leq C(n,c)\max_{Q\in\mathscr{W}_k}\mu\left(S(Q)\right)^{1/(n-1)}\mu(A_k)\\
&\leq C(n,c,D)h(2^{-j})^{1/(n-1)}\mu(A_k).
\end{align*}
Putting these estimates together gives
\begin{align}\label{finalmeas}
\mu(A_k)k\leq C \mu(A_k)^{(n-1)/n}\left(\sum_{j=1}^\infty h(2^{-j})^{1/(n-1)}j^{1/(n-1)}\right)^{(n-1)/n}.
\end{align}
Since $h$ satisfies \eqref{mofdgauge} we know that the sum on the right hand side is finite (and independent of $k$) and thus $\mu^{1/n}(A_k)k\leq C$ for every $k\in \bN$ and thus $\mu(A_k)\to0$ as $k\to\infty$, which is a contradiction. Thus we have $\mathcal{H}^h(A_\infty)=0$ and the claim follows. For quasiregular $f$ we do not get the term $j^{1/(n-1)}$ in \eqref{finalmeas} and thus a gauge satisfying \eqref{qrgauge} gives the desired conclusion. \end{proof}


\begin{remark}
As a corollary of Theorem \ref{thm:extension} we get that the set $E_f$ has Hausdorff dimension zero as well as the set $A_\infty$. Moreover, for a quasiregular $f$ we have $\mathcal{H}^{h_\varepsilon}(E_f)=0$ for every $\varepsilon>0$, where $h_\varepsilon(t)=1/\log^{n-1+\varepsilon}(1/t)$. Also, if $f$ has exponentially integrable distortion, then $\mathcal{H}^{h_\varepsilon}(E_f)=0$ for every $\varepsilon>1$. In the quasiregular case this means that $E_f$ is of conformal capacity zero in terms of gauge functions, see \cite{Wa}.
\end{remark}


\subsection{Uniqueness of limits along John curves}\label{subsec:uniqueness limits}
Assume that $\Omega$ is a $c_0$-John domain. By the previous section we have a natural candidate for the extension of $f\in\mathcal{F}_{\lambda,\mathcal{K}}(\Omega)$ to $\partial\Omega$, that is $\bar{f}\colon\partial\Omega\to\bR^n$ so that
\begin{equation}
\label{extension}\bar{f}(w)=\lim_{t\to 0^+}f(\gamma_w(t)),
\end{equation}
where $\gamma_w$ is a $c_0$-John-curve connecting $x_0$ to $w$. The problem with this extension is that it might not be well-defined even for a smooth mapping $f:\Omega\to \bR^n$.
\begin{example}
Let $\Omega=\mathbb{D}\setminus L$, where $L$ is the (euclidean) line segment connecting $-\frac{1}{2}e_1$ to $\frac{1}{2}e_1$. Then $\Omega\subset\bR^2$ is a $c$-John domain with $x_0=-\frac{3}{4}e_1$. Set $A=B(0,1/4)\cap\{x_2>0\}$ and $B=B(0,1/4)\cap\{x_2<0\}$. There is a smooth mapping $f\colon\Omega\to\bR^2$ such that $f(A)=A+i$ and $f(B)=B-i$ and thus the extension $\bar{f}$ is not well-defined for any point in $(t,0)\subset\partial\Omega$, where $t\in(-1/4,1/4)$.
\end{example}

\begin{center}
	\begin{figure}[here]
	\input{counterjohn}
		\caption{John is not sufficient}
		\label{fig:counterjohn}
	\end{figure}
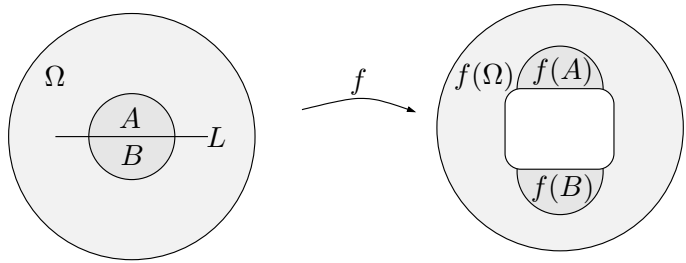
\end{center}

In order to exclude such domains we have to impose more conditions on $\Omega$. Examples like this do not work if we strengthen our assumption from John condition to $\Omega$ being uniform. 
\begin{definition}
A bounded domain $\Omega\subset \bR^n$ is a $c$-uniform domain if there exists a constant $c>0$ such that each pair of points $x_1,x_2\in \Omega$ can be joined by a $c$-uniform curve $\gamma$ in $\Omega$, namely a curve $\gamma$ with the following two properties:
\begin{align}\label{eq:uniform domain 1}
l(\gamma)\leq c|x_1-x_2|,
\end{align}
\begin{align}\label{eq:uniform domain 2}
\min_{i=1,2} l(\gamma(x_i,x))\leq c d(x,\bdary\Omega) \quad \text{for all } x\in \gamma.
\end{align}

\end{definition}
These domains first appeared in the work of Gehring and Osgood \cite{go79}. Later these domains were used in many different contexts such as Sobolev extension domains and uniformizing Gromov hyperbolic spaces, for these see~\cite{j81,bhk01}. It should be noticed that quasihyperbolic geodesics between two points in a $c$-uniform domain are $c_1$-uniform curves with $c_1=c_1(n,c)$~\cite{go79}. We say that $\Omega$ is a John (or uniform) domain if it is $c$-John (or $c$-uniform) for some $c>0$.

We may assign for each $c$-uniform domain a center point $x_0$. Indeed, if $\Omega$ is a $c$-uniform domain, then we may fix a point $x_0\in \Omega$ such that $\Omega$ is a $c_1$-John domain with center $x_0$. Moreover, the John constant $c_1$ depends only on $c$ \cite[section 2.17]{Vai}. Since we are studying boundary extension of mappings, it is convenient to introduce the following concept.

\begin{definition}[Uniform domain with center]\label{def:uniform domain with center}
We say that $\Omega$ is a $c_0$-uniform domain with center $x_0$ if $\Omega$ is a uniform domain and as a John domain, it is $c_0$-John with center $x_0$. 
\end{definition}

As was observed in Section~\ref{subsec:exist limits}, for a given point $x\in\partial\Omega$ in a John domain $\Omega$, there are typically several John curves connecting $x$ to $x_0$. Thus it might happen that $f\in\mathcal{F}_{\lambda,\mathcal{K}}(\Omega)$ has different limits at $w\in\partial\Omega$ along different John curves in $\textup{I}\mathopen(w,x_0\mathclose)$; see for instance Figure~\ref{fig:counterjohn}. In the rest of this section, we will show that if $\Omega$ is $c_0$-uniform domain with center $x_0$, then mappings in $\mathcal{F}_{\lambda,\mathcal{K}}(\Omega)$ can be extended to $\partial\Omega$ in a unique way along curves in $\textup{I}\mathopen(w,x_0\mathclose)$.


\begin{theorem}\label{thm:uniqueness of limit in Lipschitz domains}
Let $\Omega\subset \bR^n$ be a $c_0$-uniform domain with center $x_0$, $f\in\mathcal{F}_{\lambda,\mathcal{K}}(\Omega)$ and $h$ a doubling gauge function satisfying \eqref{mofdgauge}. Then $f$ has a unique limit along curves $\gamma\in \textup{I}\mathopen(w,x_0\mathclose)$ for $\mathcal{H}^h$-a.e. $\xi\in\bdary\Omega$, i.e. if $\gamma,\eta\in\textup{I}\mathopen(\xi,x_0\mathclose)$ so that 
\begin{align*}
\lim_{t\to 0^+} f(\gamma(t))=a \text{ and } \lim_{t\to 0^+} f(\eta(t))=b.
\end{align*}
Then $a=b$. 
\end{theorem}

We need the following two lemmas to prove Theorem~\ref{thm:uniqueness of limit in Lipschitz domains}.

\begin{lemma}\label{lemma:product tends to zero}
Let $\Omega$ be a $c_0$-John domain with center $x_0$ and $c\geq c_0$. If $f$ has exponentially integrable distortion and $F_f$ is the set of points $\xi\in\partial\Omega$ for which there exists a curve $\gamma\in \textup{I}_c(\xi, x_0)$ so that 
\begin{align*}
	\sum_{i=1}^\infty a_i=\infty,
\end{align*}
then $\mathcal{H}^h(F_f)=0$, where $h$ is as in \eqref{mofdgauge}.
Above, $a_i$ is defined in the following way: for a fixed $\gamma\in \textup{I}_c(\xi, x_0)$ connecting $x_0$ to $\xi \in\partial\Omega$, we denote by $Q_i\in \mathcal{W}$ those cubes  that intersect $\gamma$ and labeled in the order from $x_0$ to $\xi$, and set 
\begin{align*}
	a_i=\left(\int_{\lambda_0 B_{Q_i}} J_f(x)\,dx\right)^{1/n}\log\left(1/\diam{B_{Q_i}}\right)^{1/n}.
\end{align*}

\end{lemma}
\begin{proof}
Notice that if $\xi\in F_f$, then 
$$\sum_{i=1}^\infty\left(\int_{\lambda_0 B_{Q_i}} J_f(x)\,dx\right)^{1/n}\log\left(1/\diam{B_{Q_i}}\right)^{1/n}=\infty.$$
Let $A_k$ be the collection of points $\xi\in\partial\Omega$ for which there exists $\gamma_\xi\in \textup{I}_c(\xi,x_0)$ so that labeling as in the lemma we have
$$\sum_{i=1}^\infty\left(\int_{\lambda_0 B_{Q_i}} J_f(x)\,dx\right)^{1/n}\log\left(1/\diam{B_{Q_i}}\right)^{1/n}\geq k.$$
Then $F_f\subset A_\infty$ and thus it is enough to show that $\mathcal{H}^h(A_\infty)=0$. Towards a contradiction assume the opposite $\mathcal{H}^h(A_\infty)>0$ and denote by $\mu$ the Frostman measure whose support lies in $A_k$, which satisfies $\mu(A_k)\simeq\mathcal{H}_\infty^h(A_k)\geq\mathcal{H}_\infty^h(A_\infty)$ and
$$\mu(B(x,r))\leq Ch(r)$$
for every ball $B(x,r)$. Then we may follow the proof of Theorem \ref{thm:extension} to obtain
$$\mu(A_k)k\leq C$$
with a universal constant $C>0$. Thus we have a contradiction and the lemma is proved.
\end{proof}


\begin{lemma}\label{lemma:chaining1}
Let $\Omega\subset\bR^n$ be a $c_0$-uniform domain with center $x_0$ and $\mathcal{W}$ a Whitney decomposition of $\Omega$. Then there exists a constant $C$ such that for any $s>0$ and pair $x,\tilde{x}\in\Omega$ with $d(\tilde{x},\partial\Omega)\geq s$, $d(x,\partial\Omega)\geq s$ and $\abs{x-\tilde{x}}\leq c_1s$ there exists a chain of Whitney cubes $\{Q_k\}_{k=1}^N$ connecting points $x$ and $\tilde{x}$ such that the number of cubes is uniformly bounded with respect to $s$, i.e. $N\leq C$, where $C$ depends only on $n$, $c_0$ and $c_1$ but not on $s$. Moreover if $\xi\in S(Q)$ for some $Q\in\cup_{i=1}^N Q_i$ then $\xi\in S^{c'}(\tilde{Q})$ for any $\tilde{Q}$ from this collection with $c'=c'(C,n,c_0)\geq c_0$.
\end{lemma}

\begin{proof}
Let $s>0$ and fix $x,\tilde{x}\in\Omega$ with $d(\tilde{x},\partial\Omega)\geq s$, $d(x,\partial\Omega)\geq s$ and $\abs{x-\tilde{x}}\leq c_1s$. Connect $x$ to $\tilde{x}$ by a quasihyperbolic geodesic $[x,\tilde{x}]_k$ and enumerate the cubes $Q\in\mathcal{W}$, that intersect $[x,\tilde{x}]_k$, in a such way that $x\in Q_1$, $\tilde{x}\in Q_N$ and $Q_i\cap Q_{i+1}\neq\emptyset$ for every $i$. By \cite[page 421]{kot01} we know that the number of cubes in this chain is comparable to $k_\Omega(x,\tilde{x})$ and moreover, since $\Omega$ is uniform, by \cite[Lemma 2.13]{bhk01} we have a constant $c'(n,c_0)>0$ such that
\begin{align}\label{eq:bounds on qh length}
k_\Omega(x,\tilde{x})\leq c'\log\Big(1+\frac{\abs{x-\tilde{x}}}{\min\{d(x,\bdary\Omega),d(\tilde{x},\bdary\Omega)\}}\Big).
\end{align}
The first part of the lemma follows.

Next we prove the latter part. Since $\xi\in S(Q)$ for some $Q\in\cup_{i=1}^N Q_i$ there exists a curve $\gamma\in \textup{I}(\xi,x_0)$ such that $\gamma\cap Q\neq \emptyset$. Let $y\in Q\cap\gamma$. Let $\tilde{Q}$ be any other Whitney cubes from the collection and fix $z\in \tilde{Q}$. Let $\gamma_2$ denote the piecewise linear curve from $y$ to $z$ s.t. $\vert \gamma_2\vert\subset\cup_{i=1}^N Q_i $. We define curve $\tilde{\gamma}$ as follows. We first travel from $\xi$ along $\gamma$ to $y$ and denote this by $\gamma_1$. Next, follow $\gamma_2$ from $y$ to $z$ and denote this by $\gamma_2^+$, then go back from $z$ to $y$ along $\gamma_2$ and denote this by $\gamma_2^-$. Finally, follow $\gamma$ from $y$ to $x_0$ and denote this by $\gamma_3$. In this way we obtained curve $\tilde{\gamma}=\gamma_1\star \gamma_2^+\star \gamma_2^-\star \gamma_3$ connecting $\xi$ to $x_0$. To illustrate that $\tilde{\gamma}$ is a $c'$-John curve with $c'=c'(C,n,c_0)\geq c_0$, we assume that $\tilde{\gamma}:[0,l']\to \Omega$ is parametrized by arc length and we have to verify~\eqref{John} with $c'$.

For $t\in[0,l(\gamma_1)]$ we have
$$t\leq c_0 d(\gamma(t),\partial\Omega),$$
since $\gamma\in \textup{I}(\xi,x_0)$. Next assume that $t\in [l(\gamma_1),l(\gamma_1)+2l(\gamma_2)]$. First note that all the Whitney cubes in the collection have uniformly comparable size since the number $N$ is bounded from above uniformly by $C$ and adjacent Whitney cubes have comparable size and thus we may estimate  
$$l(\gamma_2)\leq c(n,C)d(\tilde{\gamma}(t),\partial\Omega)\quad\text{and}\quad d(y,\partial\Omega)\leq c(n,C)d(\tilde{\gamma}(t),\partial\Omega).$$
Using the above estimates, we have
\begin{align*}
t\leq l(\gamma_1)+2l(\gamma_2)\leq c_0d(y,\partial\Omega)+2l(\gamma_2)\leq c(C,n,c_0)d(\tilde{\gamma}(t),\partial\Omega).
\end{align*}
Last, assume that $t\in[l(\gamma_1)+2l(\gamma_2), l(\tilde{\gamma})]$. Now we have $\vert y-\tilde{\gamma}(t)\vert\leq c_0 d(\tilde{\gamma}(t),\partial\Omega)$, which imply
$$l(\gamma_2)\leq c(n,C)d(y,\partial\Omega)\leq c(n,C,c_0)d(\tilde{\gamma}(t),\partial\Omega)$$
and thus
\begin{align*}
t\leq c_0d(\tilde{\gamma}(t),\partial\Omega)+2l(\gamma_2)\leq c(n,C,c_0)d(\tilde{\gamma}(t),\partial\Omega).
\end{align*}
\end{proof}

\begin{proof}[Proof of Theorem~\ref{thm:uniqueness of limit in Lipschitz domains}]
Let $F_f$ denote the set defined in Lemma \ref{lemma:product tends to zero}. We know that $\mathcal{H}^h(F_f)=0$. Let $\xi\in\partial\Omega\setminus F_f$ and $\gamma\in\textup{I}(\xi,x_0)$ be such that 
$$\lim\limits_{t\to 0^+} f(\gamma(t))=a\in\bR^n.$$
We will show that $f$ has the same limit along any curve $\eta\in\textup{I}(\xi,x_0)$. Fix such a curve $\eta$. Given $r>0$, let $t_r,\tilde{t}_r\in[0,1]$ be such that  $\abs{\gamma(t_r)-w}=r$ and $\abs{\eta(\tilde{t}_r)-w}=r$. Since $\gamma,\eta$ are both $c_0$-John curves this implies that
$$d(\gamma(t_r),\partial\Omega)\geq \frac{r}{c_0}~~\text{and}~~d(\eta(\tilde{t}_r),\partial\Omega)\geq \frac{r}{c_0}.$$
We now apply Lemma~\ref{lemma:chaining1} with $s=r/c_0$ and for points $\gamma(t_r)$ and $\eta(\tilde{t}_r)$ to find a finite sequence of Whitney cubes $\{Q_i\}_{i=1}^N$ with $N\leq C$.
Choose $x_i\in Q_i$  so that $x_1=\gamma(t_r)$ and $x_{N}=\eta(\tilde{t}_r)$. Then by Lemma \ref{oscilla} we have
\begin{align*}
\abs{f(\gamma(t_r))-&f(\eta(\tilde{t}_r))}\leq\sum_{i=1}^{N}\abs{f(x_{i+1})-f(x_i)}\\
&\leq C_1\sum_{i=1}^{N} \left(\int_{\lambda_0 B_{Q_i}} J_f(x)\,dx\right)^{1/n}\log\left(1/\diam{Q_i}\right)^{1/n}\\
&\leq c(C,C_1)\max_{1\leq i\leq N}\left(\int_{\lambda_0 B_{Q_i}} J_f(x)\,dx\right)^{1/n}\log\left(1/\diam{Q_i}\right)^{1/n}.\numberthis\label{eq:osc}
\end{align*}
By Lemma \ref{lemma:chaining1} there is a constant $c'\geq c_0$ s.t. $\xi\in S^{c'}(Q_i)$ for every $i=1,\ldots,N$. Note that $N$ does not depend on $r$. It follows from Lemma~\ref{lemma:product tends to zero} that we may choose $r>0$ so small that
\begin{align*}
\max_{1\leq i\leq N}\left(\int_{\sigma Q_i} J_f(x)\,dx\right)^{1/n}\log\left(1/\diam{Q_i}\right)^{1/n}< \varepsilon.
\end{align*}
This together with \eqref{eq:osc} shows that 
\begin{align*}
\abs{f(\gamma(t_r))-f(\eta(\tilde{t}_r))}\to 0
\end{align*}
as $r\to 0$. From this we may conclude that 
$$\lim_{t\to 0^+} f(\gamma(t))=\lim_{t\to 0^+} f(\eta(t))$$
and the proof is complete.
\end{proof}


Based on Theorem~\ref{thm:uniqueness of limit in Lipschitz domains}, we introduce the following definition.

\begin{definition}[Boundary extension]\label{def:natural boundary extension}
Let $\Omega$ be a $c_0$-uniform domain with center $x_0$ and $f\in \mathcal{F}_{\lambda,\mathcal{K}}(\Omega)$. We denote by $\bar{f}(\xi)$ the unique limit of $f$ at $\xi\in \bdary\Omega$ along $c_0$-John curves. 
\end{definition}
Notice that by Theorem~\ref{thm:extension} and Theorem~\ref{thm:uniqueness of limit in Lipschitz domains}, $\bar{f}$ is well-defined $\mathcal{H}^h$-a.e. in $\bdary\Omega$. Note that the unit ball $\mathbb{B}\subset \bR^n$ is a 1-uniform domain with center at the origin. According to Definition~\ref{def:natural boundary extension}, $\bar{f}(w)$ is the unique limit of $f$ along  the unique 1-John curve, which is just the line segment $\overline{ow}$, thus in this case our Definition~\ref{def:natural boundary extension} coincides with the radial extension.


\section{Weighted capacity estimates in John domains}
In this section, we prove the following weighted capacity estimates in John domains. The techniques used here are rather standard and have been used effectively in~\cite{kot01,G13}. On the other hand, they give a new proof of the key level set estimate, namely Theorem 1.2 in~\cite{Ae2}, where $\Omega$ is the unit ball  $\mathbb{B}$. 

Recall that for disjoint compact sets $E$ and $F$ in $\overline{\Omega}$, we denote by $\capacity_{\omega}(E,F,\Omega)$ the $\omega$-weighted capacity of the pair $(E,F)$:
\begin{equation*}
\capacity_\omega(E,F,\Omega)=\inf_u \int_{\Omega}|\nabla u(x)|^ndx,
\end{equation*}
where the infimum is taken over all continuous functions $u\in W^{1,n}(\Omega)$ which satisfy $u(x)\leq 0$ for $x\in E$ and $u(x)\geq 1$ for $x\in F$.
\begin{theorem}\label{thm:main thm}
Let $\Omega\subset\bR^n$ be a $c_0$-John domain with diameter one. Let $E\subset\partial\Omega$ be a compact set satisfying $\mathcal{H}^h(E)>0$ with $h(t)=\log^{-n-1}(1/t)$.
Then 
\begin{align}\label{eq:unweighted capacity estimate}
\capacity(E,Q_0,\Omega)\geq C\Big(\log\frac{1}{\mathcal{H}^1_\infty(E)} \Big)^{1-n},
\end{align}
where $Q_0$ is the central Whitney cube containing the John center $x_0$. Moreover, if $\omega:\Omega\to [0,\infty]$ be a weight which is exponentially integrable, then
\begin{align}\label{eq:weighted capacity estimate}
\capacity_{1/\omega}(E,Q_0,\Omega)\geq C\Big(\log\frac{1}{\mathcal{H}^1_\infty(E)} \Big)^{-n}.
\end{align}
The constant $C$ in~\eqref{eq:unweighted capacity estimate} and~\eqref{eq:weighted capacity estimate} depends only on the data associated to $\Omega$ and $\omega$.
\end{theorem}
\begin{proof}
We only prove the weighted estimate~\eqref{eq:weighted capacity estimate}, since the proof of the unweighted case is similar (and simpler). 

Let $u\in W^{1,n}(\Omega)\cap C(\overline{\Omega})$ be a test function for the $\frac{1}{\omega}$-weighted $n$-capacity of the pair $(Q_0,E)$. For each $x\in E$, we may fix a $c_0$-John curve $\gamma$ joining $x$ to $x_0$  in $\Omega$ and define $P(x)$ to be the collection of Whitney cubes that intersect $\gamma$. We define a subcollection $P'(x)\subset P(Q(x))$ as follows: $P'(x)=\{Q_1,Q_2,\dots\}$ consists of a chain of Whitney cubes, which continue along $P(x)$ until it reaches the last cube $Q_1$ for which $\diam Q_1\geq \frac{1}{5}\mathcal{H}^1_\infty(E)$. Since adjacent Whitney cubes $Q$ and $Q'$ have $\diam Q\leq 5\diam Q'$, we must have $\diam Q\leq \mathcal{H}^1_\infty(E)$ for all $Q\in P'(x)$.

We first assume that for every $x\in E$ we have $u_{Q_1}\leq \frac{1}{2}$ for $Q_1\in P'(x)$. We claim that
\begin{equation}\label{eq:claim 1}
\int_{\Omega}|\nabla u(x)|^n\frac{1}{\omega(x)}dx\geq C\Big(\log\frac{1}{\mathcal{H}^1_\infty(E)} \Big)^{-n}.
\end{equation}
If $x\in E$ and $P'(x)=\{Q_1, Q_2,\ldots\}$ then $u(x)=1$ and thus by Lebesgue differentiation Theorem we have $\lim_{i\to \infty}u_{Q_i}=1$. The standard chaining argument involving Poincar\'e inequality~(see for instance \cite{kot01,hk96}) gives us the estimate
\begin{equation}\label{eq:chaining and poincare}
1\lesssim \sum_{\substack{Q\in P'(x)}} \diam Q\,\dashint_{Q}|\nabla u(y)|dy.
\end{equation}
Note that for each $Q\subset \Omega$, by H\"older's inequality,
\begin{align*}
\dashint_Q|\nabla u|dy&=\dashint_Q|\nabla u|\omega^{-\frac{1}{n}}\omega^{\frac{1}{n}}dy\\
&\leq \Big(\dashint_Q |\nabla u|^n\frac{1}{\omega}dy\Big)^{\frac{1}{n}}\Big(\dashint_Q \omega^{\frac{1}{n-1}}dy\Big)^{\frac{n-1}{n}}.
\end{align*}
Notice that the function $t\mapsto \exp(\lambda t^{n-1})$ is convex and we may use Jensen's inequality to deduce that
\begin{align*}
\Big(\dashint_Q \omega^{\frac{1}{n-1}}dy\Big)^{\frac{n-1}{n}}&\leq \Big(\lambda^{-1}\log\Big(\dashint_Q \exp(\lambda\omega) dy\Big)\Big)^{\frac{1}{n}}\\
&\leq C(\lambda,L)\Big(\log\frac{1}{\diam Q}\Big)^{\frac{1}{n}},
\end{align*}
where $L=\int_\Omega \exp(\lambda \omega)dx$. Plugging the above estimate in~\eqref{eq:chaining and poincare}, we obtain
\begin{align}\label{eq:final esimite of chaining and poincare}
1\leq C(\lambda,n)\sum_{Q\in P'(x)} \Big(\int_{Q} |\nabla u|^n\frac{1}{\omega}dy\Big)^{\frac{1}{n}}\log^{\frac{1}{n}}\frac{1}{\diam Q}.
\end{align}
By our assumption $\mathcal{H}^h(E)>0$ and thus we can choose a Frostman measure $\mu$ on the compact set $E$ with growth function $h(r)=\log^{-n-1}(1/r)$ i.e.  $\mu$ is a Borel measure on $\bR^n$ which satisfies the growth condition $$\mu(E\cap B(x,r))\leq h(r)$$ for every ball $B(x,r)$.
Integrating~\eqref{eq:final esimite of chaining and poincare} over the set $E$ with respect to the Frostman measure $\mu$, we see that
\begin{equation*}
\mu(E)\leq C\int_E \sum_{Q\in P'(x)}\log^{\frac{1}{n}}\frac{1}{\diam Q}\Big(\int_{Q}|\nabla u|^n\frac{1}{\omega}dy \Big)^{1/n}d\mu(x).
\end{equation*}
We now interchange the order of summation and integration to deduce that
\begin{equation*}
\mu(E)\leq C\sum_{\substack{Q\in \mathscr{W}\\
\diam Q\leq \mathcal{H}^1_\infty(E)}}\mu(S_E(Q))\log^{\frac{1}{n}}\frac{1}{\diam Q}\Big(\int_{Q}|\nabla u|^n\frac{1}{\omega}dy \Big)^{1/n}.
\end{equation*}
Applying H\"older's inequality to the right hand side leads to the following upper bound for $\mu(E)$
\begin{multline*}\label{eq:upper bound for measure of E}
\left(\sum_{\substack{Q\in \mathscr{W}\\
\diam Q\leq \mathcal{H}^1_\infty(E)}}\mu(S_E(Q))^{\frac{n}{n-1}}\log^{\frac{1}{n-1}}\frac{1}{\diam Q}\right)^{\frac{n-1}{n}}\Big(\sum_{Q\in \mathscr{W}}\int_{Q}|\nabla u|^n\frac{1}{\omega}dy\Big)^{1/n}\\
\leq \left(\sum_{\substack{Q\in \mathscr{W}\\
\diam Q\leq \mathcal{H}^1_\infty(E)}}\mu(S_E(Q))^{\frac{n}{n-1}}\log^{\frac{1}{n-1}}\frac{1}{\diam Q}\right)^{\frac{n-1}{n}}\Big(\int_{\Omega}|\nabla u|^n\frac{1}{\omega}dy\Big)^{1/n}.\\
\end{multline*}
The sum in the previous estimate can be dealt with in the following way: Choose $j_0\in \mathbb{N}$ with $j_0\geq C\log(1/\mathcal{H}^1_\infty(E))$ so that $\diam Q\leq \mathcal{H}^1_\infty(E)$ implies that $Q\in \mathscr{W}_j$ for some $j\geq j_0$. We claim that
\begin{equation}\label{overlap}
\sum_{j=j_0}^\infty\sum_{Q\in \mathscr{W}_j}\mu(S_E(Q))^{1+\frac{1}{n-1}}\log^{\frac{1}{n-1}}\frac{1}{\diam(Q)}\leq C\mu(E)j_0^{-\frac{n}{n-1}}.
\end{equation}
The growth condition on $\mu$ and Lemmas \ref{ShadowProperties} and \ref{JohnMeas} imply that
\begin{align*}
&\sum_{j=j_0}^\infty\sum_{Q\in \mathscr{W}_j}\mu(S_E(Q))^{1+\frac{1}{n-1}}\log^{\frac{1}{n-1}}\frac{1}{\diam(Q)}\\
&\leq \sum_{j=j_0}^\infty\sum_{Q\in \mathscr{W}_j}\mu(S_E(Q))\mu(S_E(Q))^\frac{1}{n-1}\log^{\frac{1}{n-1}}\frac{1}{\diam(Q)}\\
&\leq C\sum_{j=j_0}^\infty\sum_{Q\in \mathscr{W}_j}\mu(S_E(Q))h(\diam Q)^\frac{1}{n-1}\log^{\frac{1}{n-1}}\frac{1}{\diam(Q)}\\
&\leq C\sum_{j=j_0}^\infty\sum_{Q\in \mathscr{W}_j}\mu(S_E(Q))\log^{-\frac{n}{n-1}}\frac{1}{\diam Q}\\
&\leq C\sum_{j=j_0}^\infty j^{-\frac{n}{n-1}}\sum_{Q\in \mathscr{W}_j}\mu(S_E(Q))\leq C\mu(E)\sum_{j=j_0}^\infty j^{-\frac{n}{n-1}}.\\
\end{align*}
Inequality \eqref{overlap} follows from this since the sum on the right hand side converges.
Putting together all the estimates above yields
\begin{equation*}
\mu(E)^n\leq C\mu(E)^{n-1}\Big(\log\frac{1}{\mathcal{H}^1_\infty(E)}\Big)^{-n}\int_\Omega |\nabla u|^n\frac{1}{\omega}\,dy.
\end{equation*}
It follows that
\begin{equation*}
\int_\Omega |\nabla u|^n\frac{1}{\omega}\geq C\mu(E)\log^n\frac{1}{\mathcal{H}^1_\infty(E)}.
\end{equation*}
To get \eqref{eq:claim 1} from here we just note that by concavity of $h$
$$\mu(E)\geq C \mathcal{H}_\infty^h(E)\geq C\log^{-1-n}\left(\frac{1}{\mathcal{H}_\infty^1(E)}\right)$$
and thus we have
$$\int_\Omega |\nabla u|^n\frac{1}{\omega} \geq C\Big(\log\frac{1}{\mathcal{H}^1_\infty (E)} \Big)^{-n}.$$
We are left with the case that there is $x\in E$ so that $u_{Q_1}\geq \frac{1}{2}$. In this case, we have $u_{Q_0}=0$ and $u_{Q_1}\geq \frac{1}{2}$. By following the quasihyerbolic geodesic from $x_1$ to $x_0$, we can find a finite chain of Whitney cubes $\{\tilde{Q}_j\}_{j=0}^l$ which joins $Q_0$ to $Q_f$, namely, cubes $\{\tilde{Q}_j\}_{j=1}^l$ such that $\tilde{Q}_0=Q_0$ is the central cube, $\tilde{Q}_l=Q_1$.
Moreover,
$$\diam Q_1\lesssim \min\{\diam \tilde{Q}_j\}\quad \text{and}\quad l\lesssim \log\frac{1}{\diam Q_1},$$
where the last inequality follows from the fact that the length of a quasihyperbolic geodesic is comparable with the number of Whitney cubes it intersects. Then repeating the chaining argument for finite chain of cubes $\{Q_0,\ldots,Q_1\}$ gives the estimate
\begin{align*}
1&\lesssim \sum_{i=1}^l\diam \tilde{Q}_i\dashint_{\tilde{Q}_i}|\nabla u|dy\lesssim \sum_{i=1}^l\Big(\int_{\tilde{Q}_i} |\nabla u|^n\frac{1}{\omega}dy\Big)^{\frac{1}{n}}\Big(\log\frac{1}{\diam \tilde{Q}_i} \Big)^{\frac{1}{n}}\\
&\lesssim \Big(\log\frac{1}{\diam Q_1}\Big)^{\frac{1}{n}}\sum_{i=1}^l\Big(\int_{\tilde{Q}_i} |\nabla u|^n\frac{1}{\omega}dy\Big)^{\frac{1}{n}}\\
&\lesssim  \Big(\log\frac{1}{\diam Q_1}\Big)^{\frac{1}{n}}\Big(\sum_{i=1}^l \int_{\tilde{Q}_i} |\nabla u|^n\frac{1}{\omega}dy\Big)^{\frac{1}{n}}l^{(n-1)/n},
\end{align*}
from which we obtain 
\begin{align*}
\int_\Omega |\nabla u(x)|^n\omega^{-1}(x)dx\geq C\Big(\log\frac{1}{\diam Q_1} \Big)^{-n}\geq C\Big(\log\frac{1}{\mathcal{H}^1_\infty (E)} \Big)^{-n}
\end{align*} 
since $\diam Q_1\leq \mathcal{H}^1_\infty (E)$. 
\end{proof}
In Theorem~\ref{thm:main thm} we have assumed that $\Omega$ is a John domain. However, the proof works for more general domains. More precisely, it hold if $\Omega$ satisfies the quasihyperbolic boundary condition~\cite{kot01}. The only difference is that in Theorem~\ref{thm:extension}, one has to consider limits along, instead of John curves, quasihyperbolic geodesics to the boundary.

\begin{remark}\label{rmk:on weighted capacity in quasihyperbolic boundary condition}
Note that, in Theorem~\ref{thm:main thm}, we have assumed that $\diam \Omega=1$. This assumption was only used to ensure that $\log\frac{1}{\mathcal{H}^1_\infty(E)}$ is positive. For domains with different diameter, one just needs to replace the right-hand side of~\eqref{eq:unweighted capacity estimate} with 
\begin{align*}
C\Big(\log\frac{2\diam\Omega}{\mathcal{H}^1_\infty(E)} \Big)^{1-n},
\end{align*}
where the constant $C$ in front now depends additionally on $\Omega$ as well. Similar changes applies to~\eqref{eq:weighted capacity estimate}.

Another observation one has to notice is that if we replace the central Whitney cube $Q_0$ by a ball $B_0=B(x_0,r_0)$ in the capacity estimates in Theorem~\ref{thm:main thm}, then the estimates~\eqref{eq:unweighted capacity estimate} and~\eqref{eq:weighted capacity estimate} become
\begin{align}\label{eq:unweighted capacity estimate for ball}
\capacity(E,B_0,\Omega)\geq C\Big(\log\frac{2\diam \Omega}{\mathcal{H}^1_\infty(E)} \Big)^{1-n},
\end{align}
and
\begin{align}\label{eq:weighted capacity estimate for ball}
\capacity(E,B_0,\Omega)\geq C\Big(\log\frac{2\diam \Omega}{\mathcal{H}^1_\infty(E)} \Big)^{1-n},
\end{align}
respectively, where the constant $C$ in both~\eqref{eq:unweighted capacity estimate for ball} and~\eqref{eq:weighted capacity estimate for ball} now depends additionally on $\Omega$ and $r_0$. Indeed, notice that if $u_{B_0}=0$, then $u_{Q_0}\leq m_0=m(r_0)<1$. So we may repeat the proof of Theorem~\ref{thm:main thm} by considering two cases: 
\begin{itemize}
\item For every $x\in E$, $u_{Q_1}\leq \frac{1+m_0}{2}<1$ for $Q_1\in P'(x)$;
\item There exists a point $x\in E$ so that $u_{Q_1}\geq \frac{1+m_0}{2}>m_0$.
\end{itemize}
The proof of the first case proceeds identically as before since $\frac{1+m_0}{2}<1$. In the second case, since $u_{Q_1}\geq \frac{1+m_0}{2}>m_0$ and $u_{Q_0}\leq m_0$, we may proceed again as in the proof of Theorem~\ref{thm:main thm} using the same finite chaining argument. 
\end{remark}

\section{Proofs of Theorems~\ref{thm:regularity of boundary extension for qr} and~\ref{thm:regularity of boundary extension}}
We are ready to prove the main results. The proof is similar than the one given in \cite[Section 2]{Ae2}.
\begin{proof}[Proof of Theorems~\ref{thm:regularity of boundary extension for qr} and~\ref{thm:regularity of boundary extension}]
Let $\Omega$ be a $c_0$ uniform domain with center $x_0$. We will first prove the case $f\in\mathcal{F}_{\lambda,\mathcal{K}}(\Omega)$, the proof for $f\in\mathcal{F}_K(\Omega)$ is analogous. Assume that $\nu$ is a positive Borel measure satifying \eqref{sreg} for $s\in(0,n)$. Since $\partial\Omega$ is compact it follows from \eqref{sreg} that 
$$\nu(\partial\Omega)<\infty.$$
Furthermore, \eqref{sreg} together with the finiteness of $\nu(\partial\Omega)$ imply that for any Borel set $E\subset \partial\Omega$ we have
\begin{equation}\label{compmeas}
\nu(E)\leq C_\Omega \mathcal{H}_\infty^s(E). 
\end{equation}
We want to show that there are constants $\alpha>0$ and $C<\infty$ independent of $f$ such that
\begin{align*}
\int_{\bdary\Omega}\exp\left(\alpha|\bar{f}(w)|\right)d\nu(w)\leq C,
\end{align*} 
where $\bar{f}$ is the boundary extension given in Definition \ref{def:natural boundary extension}. By Fubini's theorem 
$$\int_{\bdary\Omega}\exp\left(\alpha|\bar{f}(w)|\right)d\nu(w)=\nu(\partial\Omega)+\int_1^\infty \nu(F_u)e^{\alpha u}\,du,$$
where $F_u=\left\{w\in\partial\Omega: \vert \bar{f}(w)\vert \geq u\right\}$. Thus by \eqref{compmeas} it is enough to bound
$$\int_1^\infty \mathcal{H}_\infty^s(F_u)e^{\alpha u}\,du.$$
Define 
$$
E_t=\{x\in \Omega:\abs{f(x)}=t\}
$$ 
and
$$
\mathcal{A}_{n-1}f(E_t)=\int_{S^{n-1}(0,t)}\text{card}f^{-1}(y)\,d\mathcal{H}^{n-1}(y).
$$
Since the functions in $\mathcal{F}_{\lambda,\mathcal{K}}(\Omega)$ are equicontinuous, we can find $r_0>0$ such that $\vert f(x)\vert \leq 1$ for any $x\in B(x_0,r_0)\subset\Omega$ and $f\in\mathcal{F}_{\lambda,\mathcal{K}}$. Here $r_0$ depends only on $n,\lambda, \mathcal{K}$. Let $\Gamma_{F_u}$ denote the collection of curves $\gamma\colon[0,1]\to\mathbb{R}^n$ that connect points in $B(x_0,r_0)$ to points in $F_u$. Then by \cite[Proposition 10.2]{Ri}
$$\textup{Cap}_{1/K}(B(x_0,r_0), F_u, \Omega)\leq \textup{Mod}_{1/K}(\Gamma_{F_u}).$$
Similarly as in \cite{Ae2, PR} using the test function
$$\rho(x)=\left(\int_1^s\frac{dt}{(\mathcal{A}_{n-1}f(E_t))^\frac{1}{n-1}}\right)^{-1}\frac{\vert Df(x)\vert}{(\mathcal{A}_{n-1}f(E_t))^\frac{1}{n-1}}$$
when $\vert f(x)\vert=t\in(1,s)$ and $\rho(x)=0$ otherwise
for the weighted modulus we have the following upper bound
$$\textup{Mod}_{1/K}(\Gamma_{F_u})\leq\left(\int_1^s\frac{dt}{(\mathcal{A}_{n-1}f(E_t))^\frac{1}{n-1}}\right)^{n-1}.$$
From Theorem \ref{thm:main thm} we have a lower bound
$$\capacity_{1/K}(B(x_0,r_0),F_s,\Omega)\geq C_2\Big(\log\frac{C_1}{\mathcal{H}^s_\infty(F_s)} \Big)^{-n},$$
see also Remark~\ref{rmk:on weighted capacity in quasihyperbolic boundary condition}.
These together with the estimate
$$(s-1)^{n/(n-1)}\leq\omega_n^{1/(n-1)}\int_1^s\frac{dt}{\mathcal{(A}_{n-1}f(E_t))^\frac{1}{n-1}}$$
following from H\"older's inequality, see \cite{Ae2, PR}, imply that
$$\mathcal{H}^s_\infty(F_s)\leq C_1 \exp\left(C_2\omega_n^{1/n}(s-1)\right),$$
from which we conclude that for $\alpha<C_2\omega_n^{1/n}$ we have
$$\int_1^\infty \mathcal{H}_\infty^s(E_s)e^{\alpha s}\,ds\leq C,$$
where the upper bound does not depend on $f$.
\end{proof}

\textbf{Acknowledgements}

The authors would like to thank Academy Professor Pekka Koskela and Professor Kai Rajala for many helpful discussions on the topic. We are very grateful to the anonymous referee for many useful comments that improve our exposition.

\bibliographystyle{amsplain}

\end{document}

%% file: Topocomb
\begin{picture}(0,0)%
\includegraphics{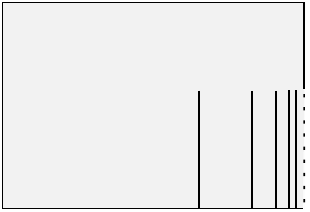}%
\end{picture}%
\setlength{\unitlength}{4144sp}%
\begingroup\makeatletter\ifx\SetFigFont\undefined%
\gdef\SetFigFont#1#2#3#4#5{%
  \reset@font\fontsize{#1}{#2pt}%
  \fontfamily{#3}\fontseries{#4}\fontshape{#5}%
  \selectfont}%
\fi\endgroup%
\begin{picture}(1408,968)(3739,-2323)
\put(3894,-1617){\makebox(0,0)[lb]{\smash{{\SetFigFont{12}{14.4}{\rmdefault}{\mddefault}{\updefault}{\color[rgb]{0,0,0}$\Omega$}%
}}}}
\end{picture}%

%% file: counterjohn
\begin{picture}(0,0)%
\includegraphics{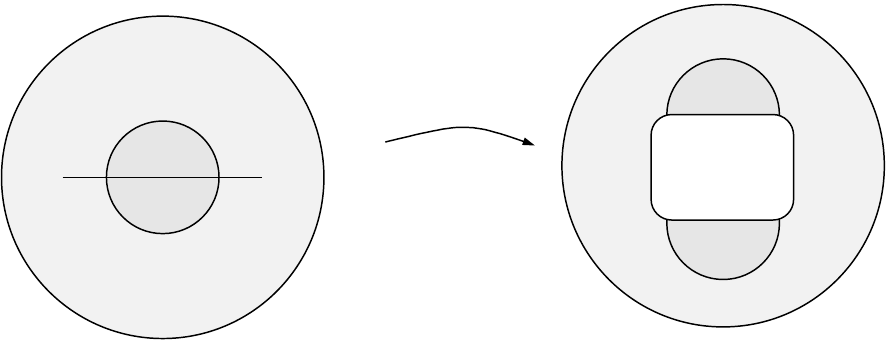}%
\end{picture}%
\setlength{\unitlength}{3729sp}%
\begingroup\makeatletter\ifx\SetFigFont\undefined%
\gdef\SetFigFont#1#2#3#4#5{%
  \reset@font\fontsize{#1}{#2pt}%
  \fontfamily{#3}\fontseries{#4}\fontshape{#5}%
  \selectfont}%
\fi\endgroup%
\begin{picture}(4501,1713)(4169,-2958)
\put(5493,-2198){\makebox(0,0)[lb]{\smash{{\SetFigFont{11}{13.2}{\rmdefault}{\mddefault}{\updefault}{\color[rgb]{0,0,0}$L$}%
}}}}
\put(6444,-1812){\makebox(0,0)[lb]{\smash{{\SetFigFont{11}{13.2}{\rmdefault}{\mddefault}{\updefault}{\color[rgb]{0,0,0}$f$}%
}}}}
\put(4415,-1795){\makebox(0,0)[lb]{\smash{{\SetFigFont{11}{13.2}{\rmdefault}{\mddefault}{\updefault}{\color[rgb]{0,0,0}$\Omega$}%
}}}}
\put(7137,-1782){\makebox(0,0)[lb]{\smash{{\SetFigFont{11}{13.2}{\rmdefault}{\mddefault}{\updefault}{\color[rgb]{0,0,0}$f(\Omega)$}%
}}}}
\put(4912,-2070){\makebox(0,0)[lb]{\smash{{\SetFigFont{11}{13.2}{\rmdefault}{\mddefault}{\updefault}{\color[rgb]{0,0,0}$A$}%
}}}}
\put(7654,-1749){\makebox(0,0)[lb]{\smash{{\SetFigFont{11}{13.2}{\rmdefault}{\mddefault}{\updefault}{\color[rgb]{0,0,0}$f(A)$}%
}}}}
\put(4920,-2321){\makebox(0,0)[lb]{\smash{{\SetFigFont{11}{13.2}{\rmdefault}{\mddefault}{\updefault}{\color[rgb]{0,0,0}$B$}%
}}}}
\put(7649,-2523){\makebox(0,0)[lb]{\smash{{\SetFigFont{11}{13.2}{\rmdefault}{\mddefault}{\updefault}{\color[rgb]{0,0,0}$f(B)$}%
}}}}
\end{picture}%